\documentclass{daj}

\usepackage{amsfonts}
\usepackage{amsthm}
\usepackage{mathtools}

\newcommand{\deck}{\mathrm{deck}}
\DeclarePairedDelimiter{\floor}{\lfloor}{\rfloor}
\DeclarePairedDelimiterX\ip[2]{\langle}{\rangle}{#1,#2}
\newcommand{\N}{\mathbb{N}}
\newcommand{\rzh}{\mathrm{rm}}
\newcommand{\Z}{\mathbb{Z}}

\newtheorem{thm}{Theorem}
\newtheorem{cor}[thm]{Corollary}
\newtheorem{lem}[thm]{Lemma}

\dajAUTHORdetails{%
  title = {Set Reconstruction on the Hypercube}, 
  author = {Luke Pebody},
  plaintextauthor = {Luke Pebody},
  keywords = {reconstruction},
}   

\dajEDITORdetails{%
   year={2017},
   number={17},
   received={3 September 2016},   
   published={27 October 2017},  
   doi={10.19086/da.2108},       
}   

\begin{document}

\begin{frontmatter}[classification=text]


\author[lp]{Luke Pebody}

\begin{abstract}
Given an action of a group $G$ on a set $S$, the $k$-deck of a subset $T$ 
of $S$ is the multiset of all subsets of $T$ of size at most $k$, each given up to 
translation by $G$.

For a given subset $T$, the {\em reconstruction number} of $T$ is the 
minimum 
$k$ such that the $k$-deck uniquely identifies $T$ up to translation by $G$, 
and the {\em reconstruction number} of the action $G:S$ is the maximum 
reconstruction number of any subset of $S$.

The concept of reconstruction number extends naturally to multisubsets $T$ 
of $S$ and in~\cite{CPC:257539}, the author calculated the 
multiset-reconstruction number of all finite
abelian groups. In particular, it was shown that the multiset-reconstruction 
number of $\Z_2^n$ was $n+1$. This provides an upper bound of $n+1$
to the reconstruction number of $\Z_2^n$. The author also showed a lower bound of $\floor{\frac{n+1}2}$ 
in the same paper.

The purpose of this note is to close the gap. The reconstruction number of 
$\Z_2^n$ is \[\floor{n+1-\log_2(n+1-\log_2(n))}.\]
\end{abstract}
\end{frontmatter}

\section{Introduction and Definitions}
Given a set $S$ and non-negative integer $k$, denote by $S^{(k)}$ the set of subsets of $S$ of size $k$. 
Given an action on a group $G$ of a set $S$, the {\em $k$-deck} 
of a subset $T$ is the multiset of all subsets of $T$ 
of size $k$, each given up to translation by 
$G$: \[\deck_{k}(T)=\{\{gU:g\in G\}:U\in T^{(k)}\}.\]

Say that subsets $T_1$ and $T_2$ are {\em $k$-indistinguishable} if 
$\deck_i(T_1)=\deck_i(T_2)$ for all $i\le k$ and 
{\em $k$-distinguishable} otherwise. In particular,
for any subset $T$ of $G$, any element $g$ of $G$ and any integer $k$, 
$gT$ and $T$ are $k$-indistinguishable and sets $T_1$ and $T_2$ are
$k$-indistinguishable if and only for all $i\le k$, there is a bijection
$\phi_i:T_1^{(i)}\to T_2^{(i)}$
such that for all $U\in T_1^{(i)}$ there exists a $g$ such that $\phi_iU=gU$.

For subsets $T_1$ and $T_2$ say that the {\em distinguishing number}
of $T_1$ and $T_2$, denoted by $d_{G:S}(T_1,T_2)$, is the 
smallest number $k$
such that $T_1$ and $T_2$ are $k$-distinguishable. If there is no such
$k$ (and so $T_1$ and $T_2$ are translates), the distinguishing
number is $\infty$.

For a subset $T$ of $S$, say that the 
{\em reconstruction number} $r_{G:S}(T)$ is 
the 
smallest number 
$k$ such that for any subset $U$ of $S$, if $T$ and $U$ are 
$k$-indistinguishable then $U$ is a translate $gT$ of $T$, and the 
{\em reconstruction number}
$r(G:S)$ is the maximum value of $r_{G:S}(T)$ for any 
subset $T\subseteq S$.

One can extend the definition of reconstruction from sets to multisets.
Let a {\em multiset} from $S$ be a map $\phi:S\to\N$ from $S$ to
the non-negative integers $\N$, and then define 
the {\em $k$-deck} of 
$\phi$ to be the map
$\deck_k\phi:S^k\to\N$ defined by
\[\deck_k\phi(s_1,s_2,\ldots,s_k)
=\sum_{g\in G}\phi(gs_1)\phi(gs_2)\ldots\phi(gs_k).\]

Given this definition of a $k$-deck, we can define $k$-distinguishable and 
$k$-indistinguishable as before, and for a multiset $\phi$ from $G$,
say that the multiset reconstruction number $\rzh_{G:S}(\phi)$ is the
smallest number $k$ such that for any multiset $\psi$ from $G$, 
if $\phi$ and $\psi$ are $k$-indistinguishable, then  $\psi$
is a translate $g\phi$ (defined by $g\phi(s)=\phi(g^{-1}s)$) of $\phi$. 
Finally, define the multiset reconstruction number $\rzh(G:S)$ to be the 
maximum value of $\rzh_{G:S}(\phi)$ for any multiset $\phi$ from $G$.

One of the first general results in this area was in~\cite{ACKR}, 
where the authors showed that the reconstruction number of
a group action $r(G:S)$ was bounded above by $\log_2(|G|)$. 

In this paper we will focus on the action of an abelian group acting 
on itself by multiplication. For the specific case of the cyclic group acting on itself by 
multiplication, it was proved in~\cite{RS} that the
reconstruction number of $r(Z_n)$ was at most equal to
9 times the number of prime factors of $n$ and was equal to 
3 for prime $n$. This was improved upon 
in~\cite{CPC:257539} where the multiset
reconstruction number was calculated for every abelian
group and, thereby, it was shown that 
$r(\Z_n)\le 6$.

We will focus on the group
$\Z_2^n$ which we will identify both, where necessary, with the 
$n$-dimensional vector
space over the field of 2 elements and with the power set of the
$n$ element set.
In~\cite{CPC:257539}, it was shown that
the multiset reconstruction number of $\Z_2^n$ is $n+1$, from 
which it follows that
$r(\Z_2^n)\le n+1$. In the same paper, a lower bound was
provided of $\frac{n+1}2$. In this paper we will show that the 
reconstruction number is equal to $\floor{n+1-\log_2(n+1-\log_2(n))}.$
This expression may seem a tad unwieldy, but comes from the following fact.

\begin{thm}\label{T:whythatexpression}
For positive integers $n$ and $k$, the statements
\[k\le\floor{n+1-\log_2(n+1-\log_2(n))}\]
and
\[2^{n+1-k}\ge k\]
are equivalent.
\end{thm}

\begin{proof}
Let us suppose that $t$ is the unique non-negative integer that 
$2^t+t\le n<2^{t+1}+t+1$. Then if we set $k=n-t$, we see that 
$2^{n+1-k}=2^{t+1}\ge n-t=k$, but if we set $k=n+1-t$, we
see that $2^{n+1-k}=2^t<n+1-t=k+1$. Since $2^{n+1-k}$ is decreasing in 
$k$, it follows that $2^{n+1-k}\ge k$ if and only if $k\le n-t$.

Further, if $2^t+t\le n<2^{t+1}$, then $\log_2(n)$ is between $t$
and $t+1$, so $n+1-\log_2(n)$ is between $2^t$ and $2^{t+1}$. 
Similarly, if $2^{t+1}\le n<2^{t+1}+t+1$, then $\log_2(n)$ is between
$t+1$ and $t+2$, so $n+1-\log_2(n)$ is between $2^t$ and $2^{t+1}$.

It follows that $\floor{n+1-\log_2(n+1-\log_2(n))}=n-t$.
\end{proof}

In Section~\ref{S:lower}, we will show that if $2^{n+1-k}\ge k$ then
the reconstruction number of $\Z_2^n$ is at least $k$, and in
Sections~\ref{S:upper} and~\ref{S:upper2}, we will show the converse. 

\section{Lower Bound}\label{S:lower}
To prove our lower bound, we provide a method for creating
sets which are not easily distinguishable.

\begin{thm}\label{T:construction1}
Suppose that $G$ contains a subgroup $H$, and that
set $A$ is the union of sets 
$A_1, \ldots, A_k$ such that for each $i$, 
$A_i$ is contained in a distinct coset $g_iH$ of $H$.

Suppose likewise that set $B$ is the union of sets
$B_1, \ldots, B_k$ such that for each $i$,
$B_i$ is contained in a distinct coset $g'_iH$ of $H$.

Suppose finally that 
for each $1\le i\le k$, $A\setminus A_i$ and $B\setminus B_i$
are translates. Then $A$ and $B$ are not $k-1$-distinguishable.
\end{thm}

\begin{proof}
Let $T$ be a subset of $G$ of size smaller than $k$,
and for set $S\subseteq G$ denote by $f(S)$ the number of elements
$g$ of $G$ for which $gT\subseteq S$.

Note that $f(S)$ counts the number of occurrences of the orbit 
$\{gT:g\in G\}$ of $T$ in $\deck_{|T|}(S)$. To show that $A$ and $B$
are not $k-1$-distinguishable, it is therefore sufficient to show that for any such $T$, 
$f(A)=f(B)$.

Let $t$ be the number of distinct cosets of $H$ in which 
the elements of $T$ lie. Note that $t\le |T|<k$. Then for 
any $g$, the elements of $gT$ lie
in $t$ cosets of $H$. Therefore 
each translate $gT$
contained in $A$ is contained in exactly $k-t$ 
of the sets $A\setminus A_j$, so
$(k-t)f(A)=\sum_i f(A\setminus A_i)=\sum_i f(B\setminus B_i)=(k-t)f(B)$
and hence $f(A)=f(B)$.

Since this is true for all sets $T$ of size less than $k$, it
follows that $A$ and $B$ are not $(k-1)$-distinguishable.
\end{proof}

Now we show how to find sets satisfying the properties of 
Theorem~\ref{T:construction1}
\begin{cor}\label{C:construction2}
For positive integer $k\ge 3$, suppose that abelian group $G$ contains a subgroup $H$ of 
index at least $k$.

Suppose also that subgroup $H$ has $k$ subgroups $H_1, H_2, \ldots, H_k$ with cosets 
$h_1H_1, h_2H_2, \ldots, h_kH_k$ such that the cosets have an empty intersection, but the intersection of any $k-1$ of them
is not empty.

Then the reconstruction number of $G$ is at least $k$.
\end{cor}

\begin{proof}
Since $\bigcap\limits_ih_iH_i$ is different from 
$\bigcap\limits_{i\ne j}h_iH_i$ for all $j$, it follows that
the cosets $h_iH_i$ are distinct. Further, since for all pairs
$i, j$, $h_iH_i\cap h_jH_j\ne\emptyset$ and cosets of a single subgroup are equal or 
disjoint, it follows that the subgroups $H_i$ and $H_j$ are distinct.

Choose $k$ distinct cosets 
$g_1H, g_2H, \ldots, g_kH$ of $H$ in $G$, and then for $1\le i\le k$, 
let $A_i=g_iH_i$ and let $B_i=g_ih_iH_i$. Finally let
$A$ be the disjoint union of the $A_i$ and $B$ be the disjoint union of 
the $B_i$. 

Choose $1\le i\le k$, and let $x$ be any element of 
$\bigcap\limits_{j\ne i}h_jH_j$. Then for all $j\ne i$, 
$xA_j=xg_jH_j=g_jxH_j=g_jh_jH_j=B_j$, and so
$x(A\setminus A_i)=\bigcup_{j\ne i}xA_j=\bigcup_{j\ne i}B_j=B\setminus B_i.$

Thus the sets $A_i$ and $B_i$ satisfy the condition
of Theorem~\ref{T:construction1}, and hence $A$ and $B$ are not
$k-1$-distinguishable. Further, we will show that
$A$ and $B$ are not translates.

Suppose otherwise, 
then there 
exists $g\in G$ such that $gA=B$. Since cosets of $H$ are 
preserved by
the map $x\to gx$, it follows that this map must map 
each $A_i$ to some $B_j$.
However, if $i\ne j$ then $A_i$ and $B_j$ are cosets
of different subgroups $H_i$ and $H_j$ which means 
one cannot be mapped to the other. 

Thus for each
$i$, it follows that $gg_iH_i=g_ih_iH_i$, from which it follows (since $G$ is 
abelian)
that $g\in h_iH_i$ for all $i$. This is a contradiction, as
\[\bigcap\limits_ih_iH_i=\emptyset.\]

It follows that $A$ and $B$ are not $k-1$-distinguishable,
and are not translates, so $G$ must have reconstruction number of at least $k$.
\end{proof}

To conclude we will show the existence of cosets with this
intersection property.

\begin{lem}\label{L:construction3}
The hypercube $\Z_2^{k-1}$ contains $k$ cosets of subgroups
with empty intersection, such that the intersection of any $k-1$ of them
is not empty.
\end{lem}

\begin{proof}
Representing the elements of $\Z_2^{k-1}$ as sequences $(x_1, x_2,
\ldots,x_{k-1})$ of elements of $\Z_2$, we can take our $k$ cosets
to be $\{x: x_1=0\}, \{x: x_1=x_2\}, \{x: x_2=x_3\}, \ldots, 
\{x: x_{k-2}=x_{k-1}\}$ and $\{x: x_{k-1}=1\}$.

Clearly no vector satisfies all of these conditions, but if
you remove any one condition, the rest can be satisfied.
\end{proof}

This concludes the work for the lower bound.

\begin{cor}\label{C:construction4}
If positive integers $n, k$ have $2^{n+1-k}\ge k$ then the 
reconstruction number of $\Z_2^{n}$ is at least $k$.
\end{cor}

\begin{proof}
For $k\ge 3$, this follows directly from Corollary~
\ref{C:construction2} and
Lemma~\ref{L:construction3}. $\Z_2^{k-1}$ is a subset
of $\Z_2^n$ of index $2^{n+1-k}$, so if this is at least $k$,
the conditions of Lemma~\ref{L:construction3} apply.

For $k=1$, the condition is $2^n\ge 1$, which is equivalent to 
$n\ge 0$. Clearly, for any $n$, the 0-deck does not distinguish
between $\{\}$ and $\Z_2^n$ which are distinct, so $\Z_2^n$ has
reconstruction number at least 1.

For $k=2$, the conditions is that 
$2^{n-1}\ge 2$, which is 
equivalent
to $n\ge 2$. 

For $n\ge 2$ and for any two non-zero elements $a, b$ of $\Z_2^n$ the sets 
$\{0,a\}$ and $\{0,b\}$ are non-translates that are not 1-distinguishable
so $\Z_2^n$ has reconstruction number at least 2.
\end{proof} 

\section{The Fourier Transform}\label{S:upper}
Given elements $x=(x_1,x_2,\ldots,x_n)$ and $y=(y_1,y_2,\ldots,y_n)$ 
of $\Z_2^n$, 
define $\ip xy=x_1y_1+x_2y_2+\ldots+x_ny_n$. Then, given a 
mapping $f:\Z_2^n\to\Z$ the Fourier Transform $\hat{f}:\Z_2^n\to\Z$ 
is defined by $\hat{f}(x)=\sum_{y\in\Z_2^n}f(y)(-1)^{\ip xy}$. 

Given a multiset $f$ from $\Z_2^n$ and a linear map 
$\theta:\Z_2^n\to\Z_2^k$, denote by $\theta f$ the multiset from
$\Z_2^k$ defined by
\[\theta f(x)=\sum_{z:\theta z=x}f(z).\]

For every linear map $\theta:\Z_2^n\to\Z_2^k$, there is a dual map
$\theta^*:\Z_2^k\to\Z_2^n$ with the property that for every 
$x\in\Z_2^n$ and $y\in\Z_2^k$, $\ip{\theta x}y=\ip x{\theta^* y}$. The dual
map gives a clean description of the Fourier Transform of linear images
of multisets.

\begin{lem}\label{L:duals}
Given a multiset $f$ from $\Z_2^n$ and a linear map
$\theta:\Z_2^n\to\Z_2^k$ with dual $\theta^*$, the Fourier Transform
of the image $\theta f$ is given by
\[\widehat{\theta f}(x)=\hat{f}(\theta^*x).\]
\end{lem}

\begin{proof}
\begin{align*}
\widehat{\theta f}(x)
&=\sum_{y\in\Z_2^k}\theta f(y)(-1)^{\ip xy}\\
&=\sum_{y\in\Z_2^k}\sum_{z:\theta z=y}f(z)(-1)^{\ip xy}\\
&=\sum_{z\in\Z_2^n}f(z)(-1)^{\ip x{\theta z}}\\
&=\sum_{z\in\Z_2^n}f(z)(-1)^{\ip{\theta^*x}z}\\
&=\hat{f}(\theta^*x).
\end{align*}
\end{proof}

The effect of translation on the Fourier Transform is similarly 
easy to describe.
\begin{lem}\label{L:translates}
Given a multiset $f$ from $\Z_2^n$ and an element
$z\in\Z_2^n$, the Fourier Transform of the translate $z+f$
is given by
\[\widehat{(z+f)}(x)=\hat{f}(x)(-1)^{\ip xz}.\]
\end{lem}

\begin{proof}
\begin{align*}
\widehat{(z+f)}(x)
&=\sum_{y\in\Z_2^n}(z+f)(y)(-1)^{\ip xy}\\
&=\sum_{y\in\Z_2^n}f(z+y)(-1)^{\ip xy}\\
&=\sum_{y'\in\Z_2^n}f(y')(-1)^{\ip x{z+y'}}\\
&=\sum_{y'\in\Z_2^n}f(y')(-1)^{\ip xz}(-1)^{\ip x{y'}}\\
&=(-1)^{\ip xz}\sum_{y'\in\Z_2^n}f(y')(-1)^{\ip x{y'}}\\
&=(-1)^{\ip xz}\hat{f}(x).
\end{align*}
\end{proof}

It is proved in~\cite{CPC:257539} that the reconstruction number of a given
multiset can be described
in terms of the Fourier Transform.
\begin{thm}\label{T:fourier}
The $k$-deck of $f$ gives exactly the same information as knowing all values 
$\prod_{i=1}^k\hat{f}(x_i)$ 
for all sequences $(x_1,x_2,\ldots,x_k)$ with $x_1+x_2+\ldots+x_k=0$.
\end{thm}

We can use this to show the effect of linear maps on distinguishability.

\begin{thm}\label{T:thisone}
For any multisets $f, g$ on $\Z_2^n$ and integer $k\ge 2$:
\begin{enumerate}
\item If $f$ and $g$ are $k$-indistinguishable then for any linear map $\theta$ from $\Z_2^n$ to any target space, $\theta f$ and $\theta g$
are $k$-indistinguishable.
\item If $f$ and $g$ are $k$-distinguishable, then there exists a linear
map $\theta:\Z_2^n\to\Z_2^{k-1}$ such that $\theta f$ and
$\theta g$ are $k$-distinguishable.
\end{enumerate}
\end{thm}

\begin{proof}
By Theorem~\ref{T:fourier}, $f$ and $g$ being $k$-indistinguishable is 
equivalent to 
\[\hat{f}(x_1)\ldots\hat{f}(x_i)=\hat{g}(x_1)\ldots\hat{g}(x_i)\]
for all sequences $x_1, \ldots, x_i$ with $i\le k$ and $x_1+\ldots+x_i=0$.

Suppose $f$ and $g$ are $k$-indistinguishable, and let
$\theta:\Z_2^n\to\Z_2^p$ be any linear map. Then by 
Lemma~\ref{L:duals}, for any sequence $x_1, \ldots, x_i\in\Z_2^p$ with 
$i\le k$ and $x_1+\ldots+x_i=0$,
\begin{align*}
\widehat{\theta f}(x_1)\ldots\widehat{\theta f}(x_i)
&=\hat{f}(\theta^*x_1)\ldots\hat{f}(\theta^*x_i)\\
&=\hat{g}(\theta^*x_1)\ldots\hat{g}(\theta^*x_i)\\
&=\widehat{\theta g}(x_1)\ldots\widehat{\theta g}(x_i),
\end{align*}
so $\theta f$ and $\theta g$ are $k$-indistinguishable.

Similarly, for $k$-distinguishable $f$ and $g$, there exists a sequence 
$x_1, \ldots, x_i$ with $i\le k$, $x_1+\ldots+x_i=0$ and 
\[\hat{f}(x_1)\ldots\hat{f}(x_i)\ne\hat{g}(x_1)\ldots\hat{g}(x_i).\]

Then
define map $\theta:\Z_2^n\to\Z_2^{k-1}$ in terms of its dual by 
$\theta^*e_j=x_j$ for $j<i$ and $\theta^*e_j=0$ for $j\ge i$. 
Note that $\theta^*(e_1+\ldots+e_{i-1})=x_i$.

Thus if we let $y_j=e_j$ for $j<i$ and $y_i=e_1+\ldots+e_{i-1}$, then
$y_1, \ldots,y_i$ are a sequence with $i\le k$, $y_1+\ldots+y_i=0$
and
\[\widehat{\theta f}(y_1)\ldots\widehat{\theta f}(y_i)\ne
\widehat{\theta g}(y_1)\ldots\widehat{\theta g}(y_i),\]
so $\theta f$ and $\theta g$ are $k$-distinguishable.
\end{proof}

\section{Structure of maximally indistinguishable multisets on $\Z_2^{k-1}$}
\label{S:upper2}
For this section, we will investigate the nature of pairs of
multisets from $\Z_2^{k-1}$ which have distinguishing number $k$. We
will give concrete examples of such multisets now.

For distinct non-negative integers $a, b$ and positive integers 
$a_1, \ldots, a_{k-1}$, denote by $f_{(a,b),(a_1,\ldots,a_{k-1})}$
the multiset defined by
\[
f((x_1, x_2, \ldots, x_{k-1}))=
\begin{cases}
a+\sum_ia_ix_i\textrm{ if }\sum x_i\textrm{ is even} \\
b+\sum_ia_ix_i\textrm{ if }\sum x_i\textrm{ is odd} \\
\end{cases}
\]

We note that the Fourier Transform of $f_{(a,b),(a_1,\ldots,a_{k-1})}$
takes a particularly simple form. Since the all-ones vector 
$e_1+e_2+\ldots+e_{k-1}$ will come up quite a lot this section,
denote it by $h$.

\begin{thm}\label{T:ftstandardform}
The Fourier Transform of $f_{(a,b),(a_1,\ldots,a_{k-1})}$ is given by
\[\hat{f}_{(a,b),(a_1,\ldots,a_{k-1})}(x)
=\begin{cases}
2^{k-2}(a+b+\sum a_i)\textrm{ if }x=0\\
2^{k-2}a_i\textrm{ if }x=e_i\\
2^{k-2}(a-b)\textrm{ if }x=h\\
0\textrm{ otherwise }
\end{cases}\]
\end{thm}

\begin{proof}
We can decompose $f_{(a,b),(a_1,\ldots,a_{k-1})}$ as the linear combination of 
$\frac{a+b}2$ parts of the function $x\to 1$, $\frac{a-b}2$ parts of 
the function $x\to(-1)^{\ip hx}$ and $a_i$ parts each of the function $x\to x_i$.

Each part has a Fourier Transform that is trivial to compute.

The Fourier Transform of $x\to 1$ is 
\[z\to\sum_{y\in\Z_2^{k-1}}(-1)^{\ip zy},\] which is $2^{k-1}$ for $z=0$ and
0 otherwise.

The Fourier Transform of $x\to(-1)^{\ip hx}$ is
\[z\to\sum_{y\in\Z_2^{k-1}}(-1)^{\ip hy\ip zy}=\sum_{y\in\Z_2^{k-1}}(-1)^{\ip {z+h}y},\]
which is $2^{k-1}$ for $z=h$ and 0 otherwise.

Finally the Fourier Transform of $x\to x_i$ is
\[z\to\sum_{y\in\Z_2^{k-1}}y_i(-1)^{\ip zy}.\]

If $z_j\ne 0$ with $j\ne i$ then the terms in the above expression for 
$y$ and $y+e_j$ cancel each other out, so the Fourier Transform is 0 outside of
$\{0,e_i\}$. For $z=0$ or $e_i$, the Fourier Transform is clearly $2^{k-2}$.

Combining these Fourier Transforms together gives the above form. 
\end{proof}

Then it follows that $f_{(a,b),(a_1,\ldots,a_{k-1})}$ has a twin from which
it has high distinguishing number.

\begin{cor}
For distinct non-negative integers $a, b$ and positive integers
$a_1, \ldots, a_{k-1}$, the multisets
$f_{(a,b),(a_1,\ldots,a_{k-1})}$ and
$f_{(b,a),(a_1,\ldots,a_{k-1})}$ have distinguishing number $k$.
\end{cor}

\begin{proof}
By Theorem~\ref{T:ftstandardform}, the fourier transforms
$g_1=\hat{f}_{(a,b),(a_1,\ldots,a_{k-1})}$ and 
$g_2=\hat{f}_{(b,a),(b_1,\ldots,b_{k-1})}$
have identical support $\{0, e_1, e_2, \ldots, e_{k-1}, h\}$ and
are equal except at $x=h$ for which we have 
$g_1(-x)=g_2(-x)$.

By Theorem~\ref{T:fourier}, the distinguishing number of the two
multisets is equal to the length of the minimum 0-sum sequence
$x_1, \ldots, x_t$ for which 
\[g_1(x_1)\ldots g_1(x_t)\ne g_2(x_1)\ldots g_2(x_t).\]

Clearly any such sequence must contain $h$ an odd number of times,
leaving a 1-coordinate in each location that must be made up by a copy
of $e_i$. Thus the shortest such sequence has length $k$ and is
$e_1+e_2+\ldots+e_{k-1}+h=0$.
\end{proof}

We will show that up to translation and linear maps on $\Z_2^{k-1}$,
these are the only pairs of multisets of distinguishing number $k$. Say that 
two multisets $f_1$ and $f_2$ from $\Z_2^{k-1}$ of 
distinguishing number $k$ are in {\em standard position} if 
\[\hat f_1(e_1)\ldots\hat f_1(e_{k-1})\hat f_1(h)\ne
\hat f_2(e_1)\ldots\hat f_2(e_{k-1})\hat f_2(h).\]

\begin{thm}\label{T:structure1}
For $k\ge 3$, if $f_1$ and $f_2$ are two multisets on $\Z_2^{k-1}$ 
of distinguishing number $k$, then there exists a bijective group
homomorphism $\theta:\Z_2^{k-1}\to\Z_2^{k-1}$ such that
$\theta f_1$ and $\theta f_2$ are in standard position.
\end{thm}

\begin{proof}
By Theorem~\ref{T:fourier}, there is a 0-sum sequence of length $k$,
$x_1, \ldots, x_k$ for which
\[\hat f_1(x_1)\ldots\hat f_1(x_k)\ne
\hat f_2(x_1)\ldots\hat f_2(x_k),\]
and there is no shorter sequence.

In particular this means that $x_1, \ldots, x_k$ cannot be split into
two shorter 0-sum sequences (for they would both have equal sums),
so $x_1, \ldots,x_{k-1}$ must be linearly independent and therefore
must span $\Z_2^{k-1}$. As such there is a bijective group homomorphism
$\theta:\Z_2^{k-1}\to\Z_2^{k-1}$ for which 
$\theta^*e_i=x_i$ for all $1\le i\le k-1$.

Then by Lemma~\ref{L:duals}, we have
\[\widehat{\theta f_1}(e_1)\ldots\widehat{\theta f_1}(e_{k-1})\widehat{\theta f_1}(h)\ne
\widehat{\theta g_1}(e_1)\ldots\widehat{\theta g_1}(e_{k-1})\widehat{\theta g_1}(h).\]
\end{proof}

\begin{thm}\label{T:structure}
For $k\ge 3$, if $f_1$ and $f_2$ are two multisets on $\Z_2^{k-1}$ 
of distinguishing
number $k$ in standard position, then there exist distinct non-negative integers $a, b$
, positive integers $a_1, \ldots, a_{k-1}$
and elements $x_1, x_2\in\Z_2^{k-1}$
such that
\begin{align*}
x_1 + f_1&=f_{(a,b),(a_1,\ldots,a_{k-1})}\\
x_2 + f_2&=f_{(b,a),(a_1,\ldots,b_{k-1})}
\end{align*}
\end{thm}

\begin{proof}
We know that 
\[\hat f_1(e_1)\ldots\hat f_1(e_{k-1})\hat f_1(h)\ne
\hat f_2(e_1)\ldots\hat f_2(e_{k-1})\hat f_2(h),\]
and there is no shorter such sequence. In particular, note that for 
all $x\in\Z_2^{k-1}$, $x+x=0$ is a shorter
sequence, so $\hat f_1(x)^2=\hat f_2(x)^2$.

Now let $x$ be any element of $\Z_2^{k-1}$ other than $\{0, e_1, 
e_2, \ldots, e_{k-1}, h\}$, and let $I=\{i:1\le i\le k-1, x_i=1\}$ be the set
of 1 coordinates of $x$.
Note that by the choice of $x$, $2\le p\le k-2$. As such, it follows that
$\{e_i:i\in I\}\cup\{x\}$ and $\{e_i:i\notin I\}\cup\{h, x\}$ are both
0-sum subsets of length at most $k-1$, and so it follows that
\begin{align*}
\prod_{i\in I}\hat f_1(e_i)\hat f_1(x)&=
\prod_{i\in I}\hat f_2(e_i)\hat f_2(x)\textrm{ and } \\
\prod_{i\notin I}\hat f_1(e_i)\hat f_1(h)\hat f_1(x)&=
\prod_{i\notin I}\hat f_2(e_i)\hat f_2(h)\hat f_2(x).
\end{align*}

Multiplying these together we get 
\[\prod_i\hat f_1(e_i)\hat f_1(h)\hat f_1(x)^2
=\prod_i\hat f_1(e_i)\hat f_2(h)\hat f_2(x)^2.\]

Since we know that $\hat f_1(x)^2=\hat f_2(x)^2$ and 
\[\hat f_1(e_1)\ldots\hat f_1(e_{k-1})\hat f_1(h)\ne
\hat f_2(e_1)\ldots\hat f_2(e_{k-1})\hat f_2(h),\]
it follows that $\hat f_1(x)=\hat f_2(x)=0.$

To summarize, $\hat f_1$ and $\hat f_2$ have no support outside of
$\{0, e_1, \ldots, e_{k-1}, h\}$, we have 
$\hat f_1(x)^2=\hat f_2(x)^2$ for all $x$ and
\[\hat f_1(e_1)\ldots\hat f_1(e_{k-1})\hat f_1(h)\ne
\hat f_2(e_1)\ldots\hat f_2(e_{k-1})\hat f_2(h).\]

Now let $x_1=\sum_{\hat f_1(e_i)<0}e_i$ and 
$x_2=\sum_{\hat f_2(e_i)<0}$. Then for all $i$, $\widehat{f_1+x_1}(e_i)$ and
$\widehat{f_2+x_2}(e_i)$ are positive numbers with the same square, and so
are equal.

Thus $\widehat{f_1+x_1}$ and $\widehat{f_2+x_2}$ are of the form given in
Theorem~\ref{T:ftstandardform} except that we have not 
shown that $a, b$ are necessarily non-negative, or that any of
$a, b, a_1, \ldots, a_{k-1}$ are necessarily integers.

To that end, note that $a=f_1+x_1(0)$, $b=f_2+x_2(0)$ (and so they
are indeed non-negative integers), and $a_i=f_1+x_1(e_i)-f_2+x_2(0)$, so
is integral.
\end{proof}

It follows that if we have $k$-distinguishable multisets, some element must
appear $k$ times. 

\begin{thm}\label{T:laststraw}
For integers $k\ge 2$, if multisets $f$ and $g$ on $\Z_2^{k-1}$ have
distinguishing number $k$, some element appears with multiplicity at least
$k$ in $f$ or $g$.
\end{thm}

\begin{proof}
By Theorem~\ref{T:structure}, $f$ and $g$ are translates of a linear
image of some pair $f_{(a,b),(a_1,\ldots,a_{k-1})}$ and 
$f_{(b,a),(a_1,\ldots,a_{k-1})}$. The multiplicity of $h$ in these two
multisets is $a+a_1+\ldots+a_{k-1}$ and $b+a_1+\ldots+a_{k-1}$.
Since each $a_i$ is a positive integer, and $a$ and $b$ are distinct
non-negative integers, one of these numbers is at least $k$.
\end{proof}

This now allows us to prove our upper bound.

\begin{cor}\label{C:otherhalf}
For positive integers $n$ and $k$, if the reconstruction number 
of $\Z_2^n$ is at least $k$, then $2^{n+1-k}\ge k$.
\end{cor}

\begin{proof}
Suppose that the reconstruction number of $\Z_2^n$ is $t\ge k$. 
Then there exists two sets $S_1, S_2\subseteq\Z_2^n$ that are
$t$-distinguishable, but not $t-1$-distinguishable.

Then by Theorem~\ref{T:thisone}, there exists a map 
$\theta:\Z_2^n\to\Z_2^{t-1}$ such that the multisets $\theta S_1, 
\theta S_2$ from $\Z_2^{t-1}$ are $t$-distinguishable, but not
$t-1$-distinguishable.

Finally, by Theorem~\ref{T:laststraw}, there must be some element 
of $\Z_2^{t-1}$ that appears in $\theta S_1$ or $\theta S_2$ at least
$t$ times. Since each element of $\Z_2^{t-1}$ only has
$2^{n+1-t}$ inverse images, it follows that $2^{n+1-t}\ge t$.

Since $t\ge k$, we must have $2^{n+1-k}\ge 2^{n+1-t}\ge t\ge k$.
\end{proof}

Tying everything together, we now have proved the precise
value of the reconstruction number of the hypercube.

\begin{cor}
The reconstruction number of $\Z_2^n$ is 
\[\floor{n+1-\log_2(n+1-\log_2(n))}.\]
\end{cor}

\begin{proof}
By Corollaries~\ref{C:construction4} and~\ref{C:otherhalf}, 
the reconstruction number of $\Z_2^n$ is the maximum
value of $k$ for which
if $2^{n+1-k}\ge k$. By Theorem~\ref{T:whythatexpression}, this is
$\floor{n+1-\log_2(n+1-\log_2(n))}.$
\end{proof}


\section*{Acknowledgments} 
The author is grateful to the main anonymous reviewer for many helpful observations and tips.

\bibliographystyle{amsplain}


\begin{dajauthors}
\begin{authorinfo}[lp]
  Luke Pebody\\
  Rokos Capital Management\\
  23 Savile Row\\
  London, United Kingdom\\
  luke\imageat{}pebody\imagedot{}org
\end{authorinfo}
\end{dajauthors}

\end{document}